\documentclass[reqno]{amsart}
\usepackage{graphicx, wrapfig, amssymb, cite}
\usepackage[colorlinks,plainpages,citecolor=magenta, linkcolor=blue, bookmarksnumbered]{hyperref}

\newtheorem{thm}{Theorem}

\newtheorem{lem}{Lemma}

\theoremstyle{definition}

\theoremstyle{remark}

\numberwithin{equation}{section}
\newcommand{\dmn}{\mathop{\rm dom}}
\newcommand{\rank}{\mathop{\rm rank}}
\newcommand{\spn}{\mathop{\rm span}}
\newcommand{\Ran}{\mathop{\rm Ran}}
\newcommand{\graph}{\mathop{\rm graph}}

\renewcommand{\kappa}{\varkappa}

\newcommand{\Real}{\mathbb R}

\newcommand{\Comp}{\mathbb C}
\newcommand{\eps}{\varepsilon}

\newcommand{\cA}{\mathcal{A}}
\newcommand{\cE}{\mathcal{E}}

\newcommand{\cJ}{\mathcal{J}}
\newcommand{\cK}{\mathcal{K}}
\newcommand{\cN}{\mathcal{N}}
\newcommand{\cP}{\mathcal{P}}

\newcommand{\cT}{\mathcal{T}}

\newcommand{\cQ}{\mathcal{Q}}

\newcommand{\cX}{\mathcal{X}}

\newcommand{\ve}{v^\eps}
\newcommand{\lme}{\lambda^\eps}

\newcommand{\rRz}{\mathrm{R}_z}
\newcommand{\rR}{R_\mu}
\begin{document}

\title[Spectral Problems on Open-Book Structures]{Spectral Problems on Open-Book Structures\\ with Singularly Perturbed Density:\\ The Limit Operator}%
\author{Yuriy Golovaty} \address[YG]{Faculty of Mechanics and Mathematics, Ivan Franko National University of Lviv,
	1 Universytetska st., 79602 Lviv \and Faculty of Applied Sciences, Ukrainian Catholic University, 2a Kozelnytska str., 79026, Lviv, Ukraine}%
\email{yuriy.golovaty@lnu.edu.ua \and yuriy.golovaty@ucu.edu.ua}
\author{Delfina G\'omez} \address[DG]{Departamento Matem\'aticas, Estad\'istica y Computaci\'on, Universidad de Cantabria, Av.~Los Castros, Santander, 39005, Spain}
\email{gomezdel@unican.es}
\author{Maria-Eugenia P\'erez-Mart\'{\i}nez} \address[MP]{Departamento Matem\'atica Aplicada y Ciencias de la Computaci\'on, Universidad de Cantabria, Av.~Los Castros, Santander, 39005, Spain}
 \email{meperez@unican.es}


\maketitle

\begin{abstract}
We investigate the spectral problem arising in the asymptotic analysis of vibrations of open-book structures with a mass density perturbed near
the binding. The limiting problem is go\-ver\-ned by a non-self-adjoint block operator matrix coupling the macroscopic and microscopic components of the model.  We describe the spectrum of this operator and completely characterize its eigenspaces and root subspaces. We further prove that generalized eigenvectors form chains of length at most two and derive an explicit criterion for the existence and number of Jordan blocks. This model provides a nontrivial example of a family of self-adjoint operators acting in varying Hilbert spaces whose limiting spectral behavior is described by a non-self-adjoint operator with a genuine Jordan structure.
\end{abstract}

\section{Introduction}

Boundary value problems for partial differential equations on stratified manifolds arise naturally in mathematical models of wave propagation, diffusion, elasticity,
quantum transport, and other physical processes. Such manifolds consist of components of different dimensions joined along lower-dimensional interfaces.
Their mathematical theory has undergone substantial development over the past decades, including questions of well-posedness, regularity, asymptotic analysis
near singularities, and spectral properties; see, for example, \cite{Nicaise1988, DaugeNlcalse1989, LagneseLeugering1993, NicaiseSandig1994, NicaiseVonBelow1996, NicaiseSandig1999, NicaisePenkin2004}. An important class of stratified manifolds is formed by open-book structures, consisting of a finite collection of surfaces joined along a common curve (called the \emph{pages} and the \emph{binding}, respectively) \cite{CorbinKuchment2020, Corbin2020, AkdumanKuchment2024}. Figure~\ref{FigOpenBooks}\textit{(a)} illustrates the classical case of an open binding. The binding may also be a closed curve, as shown in Fig.~\ref{FigOpenBooks}\textit{(b), (c)}. In this case, the pages may have different global topology (for example, they may be diffeomorphic to disks or
cylinders). Nevertheless, every such stratified manifold has the same local structure in a neighborhood of the binding, namely that of an open book. Boundary value problems on open-book structures may be regarded as natural higher-dimensional counterparts of differential equations on metric graphs, commonly known as quantum graphs \cite{BerkolaikoKuchmentBook}. 

\begin{figure}[t]
\includegraphics[scale=.35]{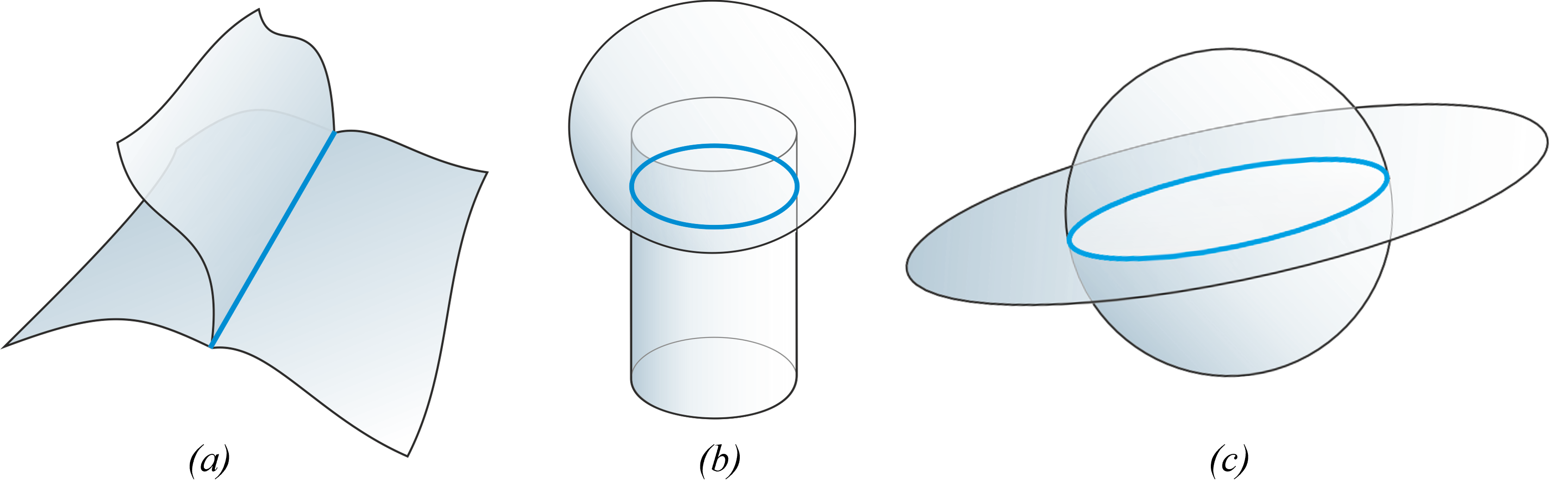}
\caption{Examples of open-book structures. The binding is highlighted in blue.}
\label{FigOpenBooks}
\end{figure}

We consider a spectral problem combining two sources of complexity: the geo\-metry of an open-book structure and a singular perturbation of the coefficients in the
governing differential equation, concentrated in a small neighborhood of the binding. This model is motivated by the study of vibrations of elastic structures
with complex geometry and highly heterogeneous mass distributions, where a substantial portion of the total mass is concentrated near the junctions.
Problems involving localized mass concentration have been extensively studied in a variety of settings, including mass concentration near points \cite{SHSPbook, OleinikBook, SPTchatat, Tchatat, GolNazOleinikSob1988, GolNazOleinikS1988, GoLoPe1999, GoLoPe-placa}, curves \cite{GoGoLoPe2, GoNaPe1, GoNaPe2, Golovaty2022AS}, and planes \cite{GoLoPemasaplano}.  A comprehensive survey of this area and the corresponding asymptotic results can be found in \cite{LoPe-review}.

A characteristic feature of vibrating systems with concentrated masses is that the asymptotic behavior of the spectrum of such systems depends critically on the strength of the mass perturbation. For weak perturbations, the leading-order spectral asymptotics are determined by the global geometry of the vibrating structure. In contrast, when the perturbation is sufficiently strong, the spectrum is governed primarily by the geometry of the shrinking neighborhood in which the concentrated mass is supported. Between these two regimes lies a critical scaling, where both mechanisms contribute simultaneously.

The critical scaling is characterized by the simultaneous influence of the macroscopic geometry of the vibrating structure and the microscopic geometry of the region where the mass perturbation is localized. As a consequence, the limit spectral problem is no longer governed by a single differential operator but takes the form of a coupled spectral system associated with a block operator matrix whose entries act on different geometric objects. Although each perturbed problem is associated with a self-adjoint operator in a weighted Hilbert space, the limit operator is, in general, non-self-adjoint. This phenomenon originates from the singular dependence of the Hilbert space metric on the perturbation parameter. As distinct eigenvalues coalesce, the corresponding eigenfunctions remain orthogonal in the weighted spaces, 
whereas, considering them in the standard $L_2$-metric, the angle between them tends to zero.

Nevertheless, the corresponding invariant subspaces converge to the root subspaces of the limit operator, giving rise to a nontrivial Jordan structure. Such spectral systems have previously been derived for strings with an attached mass concentrated in a neighborhood of a point and for membranes with mass concentrated near a curve. The corresponding block operator formulations were justified by norm resolvent convergence in the former case \cite{Golovaty2020M2} and by establishing the Hausdorff convergence of the spectra in the latter \cite{Golovaty2022AS}.

This work continues the study initiated in \cite{GGP-JMAA2025}, where an analogous spectral problem on an open-book structure with a subcritical density
perturbation was investigated and  asymptotic expansions of the eigenvalues and eigenspaces were established. In contrast to the subcritical case, where the limit operators are self-adjoint with compact resolvent, the critical regime gives rise to a non-self-adjoint operator with a considerably richer spectral structure. For this reason, the present paper is devoted exclusively to the analysis of its spectrum, eigenspaces, and root subspaces. The justification of the limit model, including the convergence of the perturbed operators and their spectra, will be the subject of a subsequent paper.

The paper is organized as follows. Section~\ref{Section2}  introduces the geometric framework and the perturbed problem with a high mass density near the binding.
Section~\ref{Section3} derives the limit spectral problem and shows that its operator realization is given by a block operator matrix. Section~\ref{Section4} establishes the spectral properties of the limit operator. Section~\ref{Section5} is devoted to the characterization of its eigenspaces, root subspaces, and Jordan structure. Finally, Section~\ref{Section6} presents an explicit model example illustrating the main results.

\section{Statement of the problem}\label{Section2}
Let $\{\Omega_k\}_{k=1}^K$ be a collection of compact two-dimensional Riemannian manifolds with Lipschitz boundaries. We assume that these manifolds are smoothly embedded into $\Real^3$ in such a way that they share a common part of their boundaries: there exists a smooth curve $\Gamma$ contained in each $\partial\Omega_k$, and any two manifolds may intersect only along $\Gamma$. Such a configuration is a particular example of a topological structure known as an ``open book'' \cite{CorbinKuchment2020}. The connecting curve  may be either open or closed. In the latter case, the pages $\Omega_k$ resemble cylinders or hemispheres (see Fig.~\ref{FigOpenBooks}).

\begin{figure}[!h]
  \centering
  \includegraphics[scale=0.28]{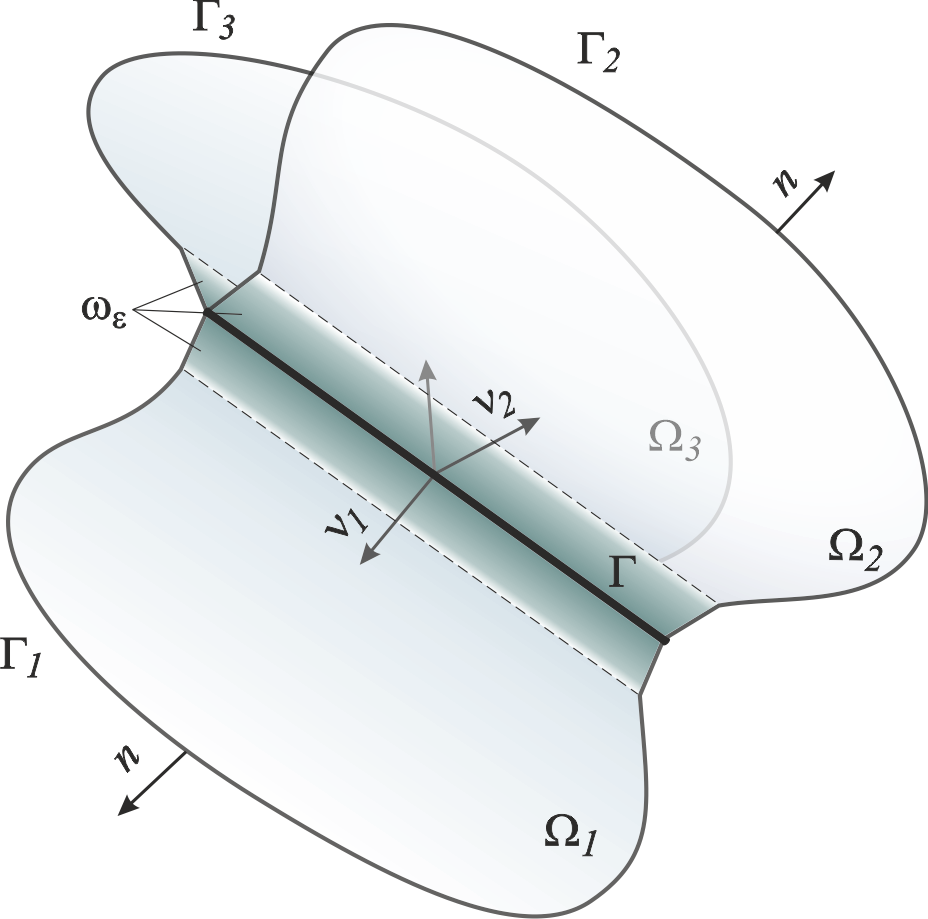}\\
  \caption{The open book structure $\Omega$: the pages $\Omega_k$ are connected along the common binding $\Gamma$, while $\nu_k$ denotes the inward normal vector field to $\Gamma$ regarded as a part of $\partial\Omega_k$.}\label{FigSS}
\end{figure}

We introduce some notation used throughout the paper. Set  
$$
   \Omega=\Omega_1\cup\ldots\cup \Omega_K\cup \Gamma, \qquad \Gamma_k=\partial\Omega_k\setminus\Gamma,
$$ 
so that $\partial \Omega=\bigcup_{k=1}^K \Gamma_k$ is the boundary of the open book structure. Let $n$ be the unit outward normal vector field to $\partial\Omega$, and let $\nu_k$ be the unit inward  normal vector field to $\Gamma$, regarded as a part of $\partial\Omega_k$ (see Fig.~\ref{FigSS}).

A function $\phi$ on $\Omega$ is understood as a vector-valued function
\begin{equation*}
  \phi=(\phi_1,\phi_2,\dots,\phi_K), \qquad \phi_k\colon \Omega_k\to \Comp.
\end{equation*}
Since the traces $\phi_k|_\Gamma$ generally do not coincide, we do not assign pointwise values to $\phi$ on $\Gamma$. Instead, we write
\begin{equation*}
  \phi|_\Gamma=(\phi_1|_\Gamma,\phi_2|_\Gamma,\dots,\phi_K|_\Gamma)
\end{equation*}
for the vector of traces on $\Gamma$.

We study the spectral problem
\begin{align}\label{PertPrblEq}
 -&\Delta_\Omega \ve+b\ve=\lme\rho^\eps \ve \quad\text{in }\Omega\setminus\Gamma,\\\label{PertPrblNC}
 &\ell\ve=0\quad\text{on } \partial\Omega,\\\label{PertPrblCont}
 &\ve_1=\ve_2=\cdots=\ve_K\quad\text{on } \Gamma,\\\label{PertPrblKirh}
 &\partial_{\nu_1}\ve_1+\partial_{\nu_2}\ve_2+\cdots+\partial_{\nu_K}\ve_K=0 \quad\text{on } \Gamma
\end{align}
with a heterogenous  mass density $\rho^\eps$, depending on a small positive parameter $\eps$.  The operator $\Delta_\Omega$ acts as the Laplace-Beltrami operator $\Delta_{\Omega_k}$ on each manifold  $\Omega_k$, and the potential $b$ is a real-valued function from  $L^\infty(\Omega)$. The notation $\ell \phi=0$ stands for the collection of boundary conditions 
\begin{equation*}
  \big\{\ell_k \phi_k=0\;\;\text{on } \Gamma_k\big\}_{k=1}^K,
\end{equation*}
where each condition $\ell_k v=0$ is of Dirichlet, Neumann, or Robin type on $\Gamma_k$. Moreover, different boundary conditions may be imposed on different parts of $\partial\Omega$. 
In \eqref{PertPrblKirh}, $\partial_{\nu_k}$ denotes the inward normal derivative to $\Gamma$ on the $k$th page (see Fig.~\ref{FigSS}).

Problem \eqref{PertPrblEq}--\eqref{PertPrblKirh} can be regarded as the system
\begin{equation*}
  -\Delta_{\Omega_k}\ve_k+b_k\ve_k=\lme\rho^\eps_k\ve_k   \quad\text{in }\Omega_k,\qquad k=1,\dots,K,
\end{equation*}
where each function $\ve_k$ satisfies the boundary conditions \eqref{PertPrblNC} on the outer boundary together with the transmission
conditions \eqref{PertPrblCont}, \eqref{PertPrblKirh} on the common binding~$\Gamma$.
The latter are not imposed ad hoc. They arise as the effective coupling conditions in the asymptotic analysis of  fattened three-dimensional open-book structures as their thickness tends to zero; see \cite{CorbinKuchment2020,Corbin2020,AkdumanKuchment2024}.
Condition \eqref{PertPrblCont} ensures continuity of the displacement across the binding, whereas condition \eqref{PertPrblKirh} expresses the balance of normal forces acting on the connected membranes.

We describe the local structure of $\Omega$ near the binding and specify the dependence of the mass density $\rho^\eps$ on the small parameter $\eps$.
Let $G=(V,E)$ be a star graph with the vertex set  
$
    V=\{a_0,a_1,\dots,a_K\}
$  
and the edge set 
$$
    E=\{e_1=(a_0,a_1),\dots,e_K=(a_0,a_K)\},
$$ 
so that all edges are incident to the central vertex $a_0$. We realize $G$ as a planar metric graph by embedding it into an auxiliary plane $\Real_t^2$ endowed with the induced Euclidean metric. We assume that the  vertex $a_0$ coincides with the origin, while the  vertices $a_1,\dots,a_K$ lie on the unit circle $S^1$. The edge $e_k$ is identified with the radius connecting the origin to the vertex $a_k$.(see Fig.~\ref{FigGandLocalOBS}).

Let $\alpha\colon [0,l)\to \Real^3$ be a smooth parametrization of $\Gamma$ by the arc-length parameter $s$. We consider the trivial normal bundle
\begin{equation*}
    N(\Gamma)=\bigcup_{x\in\Gamma}N_x
\end{equation*}
associated with $\Gamma$. For each point $x=\alpha(s)\in\Gamma$, the normal space $N_x$ is the two-dimensional linear subspace of $\Real^3$ orthogonal to
the tangent vector $\dot{\alpha}(s)$.

\begin{figure}[b]
  \centering
  \includegraphics[scale=0.45]{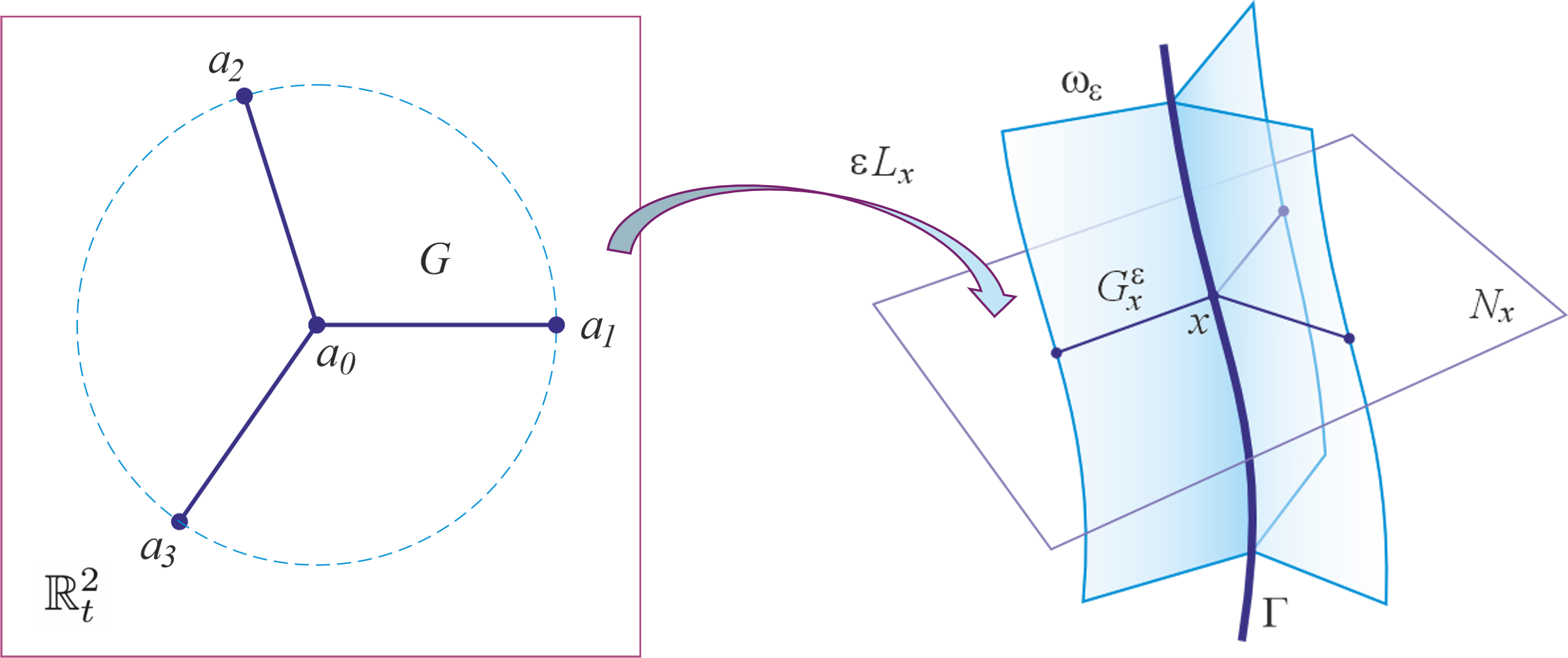}\\
  \caption{Local geometry near the binding $\Gamma$: the auxiliary star graph $G$ and its scaled image $G_x^\eps$ in the normal plane $N_x$.}\label{FigGandLocalOBS}
\end{figure}

Let $i_x\colon\Real_t^2\to N_x$ be the linear isometry sending the standard
basis of $\Real_t^2$ to the normal and binormal vectors of the Frenet--Serret
frame at $x$. Define $G_x^\eps=\eps\,i_xG$, that is, the scaled image of $G$ in the normal plane $N_x$
(see Fig.~\ref{FigGandLocalOBS}). The intersection of
$\Omega$ with the $\eps$-neighbourhood of $\Gamma$ is
\begin{equation*}
\omega^\eps=\bigcup_{x\in\Gamma}G_x^\eps.
\end{equation*}
Hence, $\omega^\eps$ is an open-book structure of thickness $O(\eps)$. For sufficiently small $\eps$, the page $\omega_k^\eps=\omega^\eps\cap\Omega_k$ is a smooth manifold. On $\omega_k^\eps$ we introduce local coordinates $(s,r_k)$, where $s$ is the arc-length parameter along $\Gamma$ and $r_k$ denotes the distance to $\Gamma$ measured within $\Omega_k$. If the curve $\Gamma$ is open, we identify the parameter domain for $s$ with the interval $[0,l]$. In the closed case, we identify $s$ with the coordinate on the circle of length $l$. Then $\omega^\eps$ is homeomorphic to the product $\Gamma\times G^\eps$.
We also introduce the notation
\begin{equation*}
\omega=\Gamma\times G,
\qquad \gamma_k=\Gamma\times\{a_k\},\qquad
\gamma=\Gamma\times\partial G=\gamma_1\cup\dots\cup\gamma_K,\qquad
\end{equation*}
where $\partial G=\{a_1,\dots,a_K\}$.

We are now in a position to specify the dependence of the mass density $\rho^\eps$ on the small parameter $\eps$. Let $\rho\colon\Omega\to\Real$ and $q\colon\omega\to\Real$ be positive functions from $L^\infty$. We assume that the perturbation of the mass density is independent of the longitudinal variable $s$ and varies only in the directions normal to the binding $\Gamma$. More precisely,
\begin{equation*}
q_k^\eps(x)=q_k(r_k/\eps), \qquad x\in\omega_k^\eps.
\end{equation*}
This assumption is essential. Allowing $q$ to depend on $s$ leads to a different limit spectral problem and requires different asymptotic techniques; see, for example, \cite{Nazarov2002, FriedlanderSolomyak2009}. We do not consider this case here.

The perturbed density is defined by
\begin{equation*}
  \rho^\eps=
  \begin{cases}
   \phantom{\eps^{-m}} \rho &\text{in } \Omega\setminus \omega^\eps,\\
    \eps^{-2}q^\eps &\text{in } \omega^\eps.
  \end{cases}
\end{equation*}
Thus, a substantial part of the total mass is concentrated in a neighbourhood of the binding. Such heterogeneous mass distributions naturally arise in models of vibrating mechanical structures with complex geometry, for instance due to welding or other metal joining techniques (see Fig.~\ref{FigWelding}).

\begin{figure}[b]
  \centering
  \includegraphics[scale=0.45]{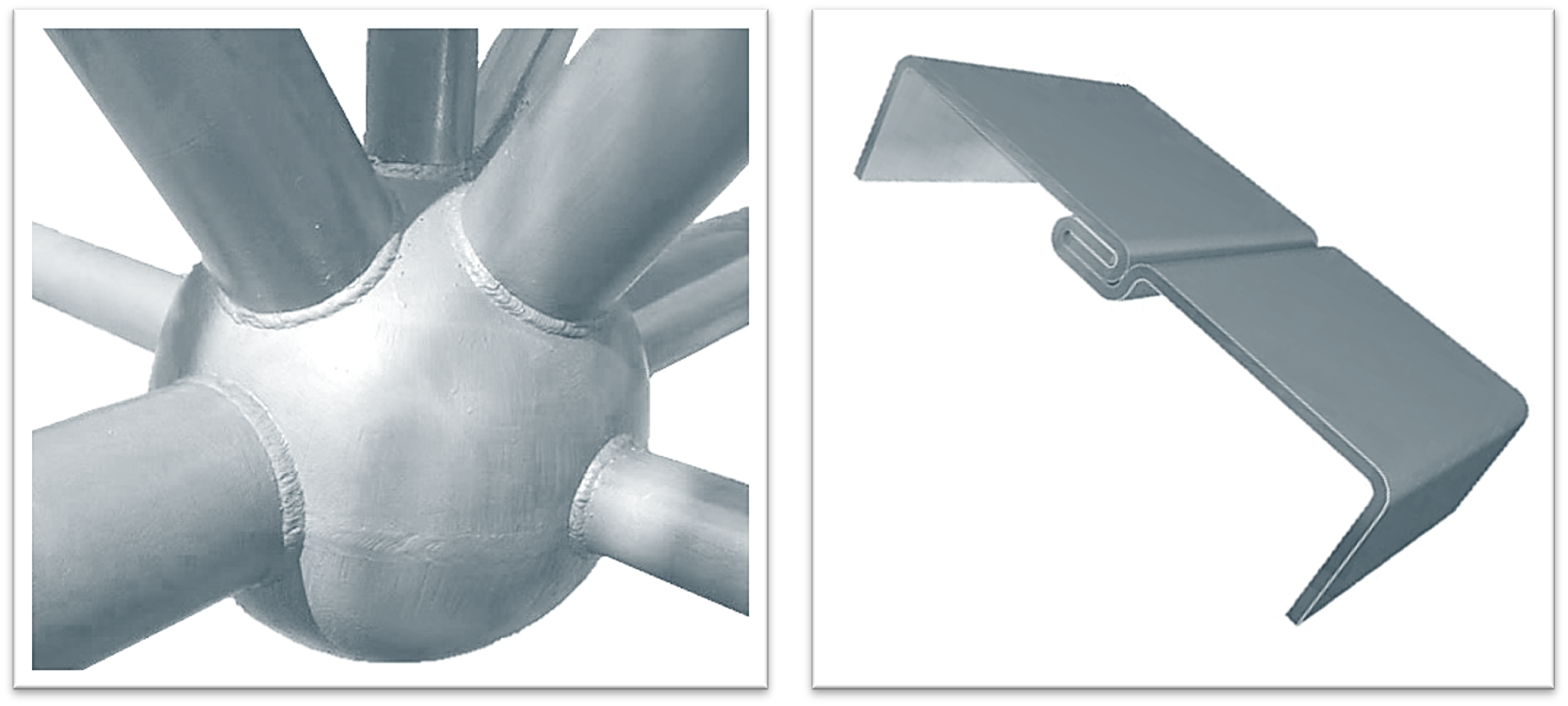}\\
  \caption{Examples of techniques for joining metal leading to hete\-ro\-geneous mass concentration near junctions: welding seams and folding tabs.}\label{FigWelding}
\end{figure}

Let $L_2(\rho^\eps,\Omega)$ be the weighted space of square-integrable functions equipped with the norm
\begin{equation*}
 \|\phi\|_\eps=\left(\int_\Omega \rho^\eps|\phi|^2\,d\mu\right)^{1/2},
\end{equation*}
where $d\mu$ denotes the Lebesgue measure on $\Omega$.  We say that a function $\phi$ belongs to the space $X(\Omega)$ if $\phi_k\in X(\Omega_k)$ for all $k=1,\dots,K$. Thus,
$
  X(\Omega)=\bigoplus_{k=1}^{K}X(\Omega_k),
$
and the corresponding norm is defined by
$$
\|\phi\|_{X(\Omega)}=\sum_{k=1}^{K}\|\phi_k\|_{X(\Omega_k)}.
$$
Throughout the paper, $W_2^j(\Omega)$ denotes the Sobolev space consisting of functions from $L_2(\Omega)$ whose weak derivatives up to order $j$ also belong to $L_2(\Omega)$.

The spectral problem \eqref{PertPrblEq}--\eqref{PertPrblKirh} is associated with the  operator
\begin{equation*}
  A_\eps=(\rho^\eps)^{-1}\big(-\Delta_\Omega+b\big)
\end{equation*}
acting in $L_2(\rho^\eps,\Omega)$ on the domain
\begin{multline*}
    \dmn A_\eps=\big\{\phi\in W_2^2(\Omega)\colon \ell \phi=0 \text{ on } \partial\Omega,\\ \phi_1=\phi_2=\cdots=\phi_K\text{ and } \textstyle\sum_{k=1}^K\partial_{\nu_k}\phi_k=0 \text{ on } \Gamma \big\}.
\end{multline*}
As shown in \cite{GGP-JMAA2025}, the operator $A_\eps$ is self-adjoint, bounded from below, and possesses compact resolvent. It follows that the spectrum of $A_\eps$ is real and discrete, and all eigenvalues have finite multiplicity.

\section{Limit spectral problem}\label{Section3}

We introduce the stretched coordinates $(s,t_k)$ on the manifold $\omega_k=\Gamma\times e_k$ in analogy with the local coordinates $(s,r_k)$ previously defined on $\omega_k^\eps$. Here 
$$
    t_k=\eps^{-1}r_k
$$
is the rescaled transverse variable, which coincides with the natural parameter along the edge $e_k$ of the graph $G$.
We  say that $\omega$ is equipped with the coordinates $(s,t)$, meaning that each page $\omega_k$ carries its own coordinate system $(s,t_k)$  (see Fig.~\ref{FigDecomposition}). Accordingly, for a function $f$ defined on $\omega$, we write
\begin{equation*}
  f(s,t)=(f_1(s,t_1),f_2(s,t_2),\dots, f_K(s,t_K)).
\end{equation*}

We seek asymptotic approximations, as $\eps\to0$, for the eigenvalues $\lambda^\eps$ of $\cA_\eps$ and the corresponding eigenfunctions $\ve$. Assume that
\begin{equation*}
\lambda^\eps=\lambda+o(1),
\end{equation*}
and that the eigenfunctions admit different asymptotic descriptions outside and inside the shrinking neighbourhood $\omega^\eps$. Namely,
\begin{equation*}
\ve(x)=v(x)+o(1),
\qquad
x\in\Omega\setminus\omega^\eps,
\end{equation*}
whereas
\begin{equation*}
\ve(x)=u(s,\eps^{-1}r)+o(1),
\qquad
x=(s,r)\in\omega^\eps.
\end{equation*}
At least one of the limit functions, $v$ or $u$, is assumed to be nontrivial.

Since the neighbourhood $\omega^\eps$ collapses onto the curve $\Gamma$ as $\eps\to0$, the outer approximation $v$ formally satisfies the equation
\begin{equation*}
-\Delta_\Omega v+bv=\lambda\rho v \quad\text{in }\Omega\setminus\Gamma.
\end{equation*}
Moreover, $v$ inherits  boundary condition \eqref{PertPrblNC} together with certain transmission conditions along $\Gamma$. To identify these conditions, we examine equation \eqref{PertPrblEq} in the neighbourhood $\omega^\eps$ using the stretched coordinates $(s,t)$.
In these coordinates, equation \eqref{PertPrblEq} takes the form
\begin{equation*}
\eps^{-2}\bigl(\partial_t^2+\lambda^\eps q(t)\bigr)\ve +\text{lower-order terms}=0\quad\text{in }\omega,
\end{equation*}
where $\partial_t =(\partial_{t_1},\dots,\partial_{t_K})$ denotes differentiation along the edges of the graph $G$. Passing to the limit as $\eps\to0$, we conclude that the inner approximation $u$ satisfies
\begin{equation*}
\partial_t^2u+\lambda qu=0
\quad\text{in }\omega,
\end{equation*}
together with the Kirchhoff-type junction conditions on $\Gamma$:
\begin{equation*}
u_1(s,0)=\cdots=u_K(s,0),
\qquad
\partial_{t_1} u_1(s,0)+\dots+\partial_{t_K} u_K(s,0)=0.
\end{equation*}

\begin{figure}[b]
  \centering
  \includegraphics[scale=0.32]{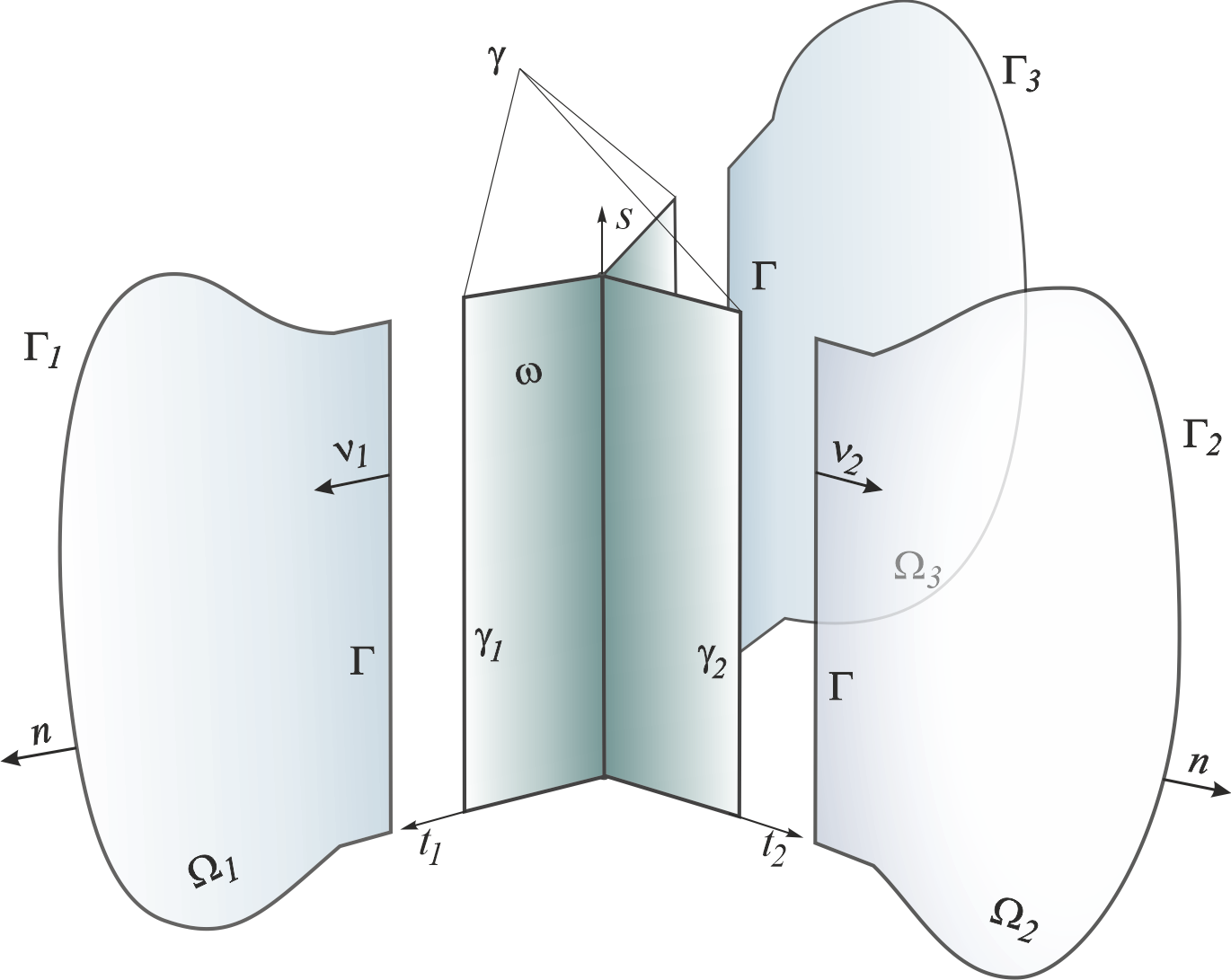}\\
  \caption{The matching interface between the outer and inner expansions. The inner domain $\omega=\Gamma\times G$ is attached to the pages
$\Omega_1,\dots,\Omega_K$ along the disconnected set $\gamma=\gamma_1\cup\dots\cup\gamma_K$.}\label{FigDecomposition}
\end{figure}

To match the inner and outer expansions, we express the outer approximation $v$ in local coordinates. Since each edge
$e_k$ has unit length, the connected component $\gamma_k$ of $\gamma$ corresponds to the cross-section $t_k=1$ in the stretched coordinates. Therefore,
\begin{equation*}
v_k(s,\eps)=u_k(s,1)+o(1),
\end{equation*}
and
\begin{equation*}
\partial_{\nu_k}v_k(s,\eps)=\eps^{-1}\partial_{t_k} u_k(s,1)+O(1)
\end{equation*}
as $\eps\to0$. 
By construction, differentiation with respect to $t_k$ coincides with differentiation along the inward normal vector $\nu_k$. Consequently, we obtain the relations
\begin{gather}\label{u=v}
v_1(s,0)=u_1(s,1),\quad v_2(s,0)=u_2(s,1),\quad \dots,\quad v_K(s,0)=u_K(s,1),
\\
\partial_{t_1}u_1(s,1)=0,\quad \partial_{t_2}u_2(s,1)=0,\quad \dots,\quad \partial_{t_K}u_K(s,1)=0.
\label{partialV}
\end{gather}
For brevity, we rewrite conditions \eqref{partialV} in the compact form
\begin{equation*}
\partial_t u=0
\quad\text{on }\gamma.
\end{equation*}
Relations \eqref{u=v} can be expressed as $v|_\Gamma=u|_{\gamma}$, where
\begin{equation*}
u|_\gamma=\bigl(u_1|_{\gamma_1},\dots,u_K|_{\gamma_K}\bigr)=\bigl(u_1(s,1),\dots,u_K(s,1)\bigr).
\end{equation*}

Collecting the above relations, we obtain the \textit{limit spectral problem:}
\begin{align}\label{LimitEqV}
-&\partial_t^2u=\lambda qu \quad\text{in }\omega, \qquad \partial_t u=0\quad\text{on }\gamma,
\\\label{KirchV}
& u_1=\cdots=u_K, \qquad \partial_{t_1}u_1+\dots+\partial_{t_K}u_K=0 \quad\text{on }\Gamma,
\\\label{LimitEqU}
-&\Delta_\Omega v+bv=\lambda\rho v \quad\text{in }\Omega\setminus\Gamma,
\qquad \ell v=0 \quad\text{on }\partial\Omega, \qquad v|_\Gamma=u|_\gamma.
\end{align}
The above problem is a coupled spectral system posed on two open-book
structures, $\Omega$ and $\omega$, and serves as the limit spectral model for
problem \eqref{PertPrblEq}--\eqref{PertPrblKirh}. Motivated by the analogous
results for strings \cite{Golovaty2020M2} and membranes
\cite{Golovaty2022AS}, we investigate its spectral properties.

\section{Spectrum of the limit operator}\label{Section4}

Let us introduce the anisotropic Sobolev space
\begin{equation*}
W_2^{2,0}(\omega) = \left\{ \phi\in L_2(\omega)\colon
\partial_t\phi\in L_2(\omega),\; \partial_t^2\phi\in L_2(\omega)\right\},
\end{equation*}
together  with its subspace
\begin{equation*}
\cK= \Big\{\phi\in W_2^{2,0}(\omega)\colon \phi_1=\phi_2=\cdots=\phi_K, \quad \sum_{k=1}^K\partial_{t_k}\phi_k=0
\text{ on }\Gamma \Big\}
\end{equation*}
whose elements satisfy the Kirchhoff conditions on $\Gamma$. We define the operators
\begin{align}
&
\begin{aligned}\label{OperatorT}
 T=-q^{-1}\partial_t^2 \quad&\text{in }L_2(q,\omega),\\
&\dmn T =\left\{ u\in W_2^{2,0}(\omega)\colon u\in\cK
\text{ and } \partial_tu=0 \text{ on }\gamma\right\},
\end{aligned}
\\[1mm]
&\begin{aligned}\nonumber
  \mathring{S}=\rho^{-1}(-\Delta_\Omega+b) \quad\text{in }L_2(\rho,\Omega),&\\
\dmn\mathring{S}&=\Big\{v\in W_2^2(\Omega)\colon \ell v=0 \text{ on }\partial\Omega \Big\}. 
\end{aligned}
\end{align}

In the Hilbert space $L_2(q,\omega)\times L_2(\rho,\Omega)$ we consider the block operator matrix
\begin{equation*}
\cA=
\begin{pmatrix}
T & 0\\
0 & \mathring{S}
\end{pmatrix},\qquad
\dmn\cA = \left\{\begin{pmatrix}
                        u\\ v
                    \end{pmatrix}\in\dmn T\times\dmn\mathring{S}\colon u|_{\gamma}=v|_\Gamma \right\}.
\end{equation*}
The operator $\cA$ is an operator realization of   \eqref{LimitEqV}--\eqref{LimitEqU}.
 Accordingly, the limit spectral problem can be written in the form
\begin{equation*}
\cA w=\lambda w, \qquad w\in\dmn\cA.
\end{equation*}
The spectral theory of block operator matrices is developed systematically in~\cite{TretterBook}.

Our first observation concerning the operator $\cA$ is that it is non-self-adjoint. A direct computation shows that the adjoint
operator $\cA^*$ is given by
\begin{equation*}
    \cA^*=
    \begin{pmatrix}
        \mathring{T} & 0 \\
        0 & S
    \end{pmatrix},\quad
    \dmn\cA^*=\left\{\begin{pmatrix}
                        \phi\\ \psi
                    \end{pmatrix}\in \dmn \mathring{T}\times\dmn S\colon
             \partial_t \phi|_{\gamma}=\partial_\nu \psi|_\Gamma\right\},
\end{equation*}
where $\mathring{T}$ denotes the maximal extension of $T$ to  $W_2^{2,0}(\omega)$, while $S$ is the operator
\begin{multline}\label{OperatorS}
  S=\rho^{-1}(-\Delta_\Omega+b) \quad\text{in }L_2(\rho,\Omega),\\
  \dmn S=\left\{\psi\in W_2^2(\Omega)\colon \ell\psi=0 \text{ on }\partial\Omega,\quad  \psi=0 \text{ on }\Gamma \right\},
\end{multline} 
that is, the restriction of $\mathring{S}$ to the subspace $\{\psi\in \dmn\mathring{S}\colon \psi|_\Gamma=0\}$. The characterization of the adjoint operator follows by standard arguments based on Green's identity. Since the explicit form of $\cA^*$ is not needed in the subsequent analysis, we omit the proof.
In fact, the non-self-adjointness of $\cA$ is essential rather than merely formal. The operator is not similar to a self-adjoint operator and, as will be
shown below, its spectrum may contain multiple eigenvalues associated with nontrivial Jordan chains.

We use the following notation below. The spectrum and the resolvent set of an operator $A$ are denoted by $\sigma(A)$
and $\varrho(A)$, respectively. For $z\in\varrho(A)$, we write 
\begin{equation*}
\rRz(A)=(A-zI)^{-1}
\end{equation*}
for the resolvent of $A$, where $I$ is the identity operator. We also adopt the Fredholm definition of the essential spectrum (see, e.g., \cite{TretterBook}):
\begin{equation*}
\sigma_{\mathrm{ess}}(A)=\{\lambda\in\Comp\colon A-\lambda I \text{ is not Fredholm}\}.
\end{equation*}
To simplify the notation, we omit the transpose symbol and write $w=(u,v)$ instead of $w=(u,v)^t$  whenever no confusion is possible.

\smallskip
\begin{thm}\label{ThmSpectrumA}
Let $T$ and $S$ be the operators defined by \eqref{OperatorT} and \eqref{OperatorS}, respectively. Then the spectrum of $\cA$ coincides with the union of the spectra of $T$ and $S$:
\begin{equation*}
\sigma(\cA)=\sigma(T)\cup\sigma(S).
\end{equation*}
\end{thm}
\smallskip

\begin{proof}
Let $\mu\in\Comp$, $f=(f_1,f_2)\in L_2(q,\omega)\times L_2(\rho,\Omega)$ and $w=(\phi,\psi)\in\dmn\cA$. 
The equation $(\cA-\mu I)w=f$ is equivalent to the system
\begin{align}\label{SystemT}
  & (T-\mu I)\phi=f_1,\\\label{SystemS}
  & (\mathring{S}-\mu I)\psi=f_2,\qquad \psi|_\Gamma=\phi|_\gamma.
\end{align}

Assume that $\mu\in \varrho(T)\cap \varrho(S)$. The first equation admits the unique solution
\begin{equation}\label{SolPhiRes} 
  \phi=\rR(T)f_1.
\end{equation}
Then  \eqref{SystemS}  can be written as the boundary value problem
\begin{equation*}
-\Delta_\Omega \psi+b\psi-\mu\rho \psi=\rho f_2 \quad\text{in }\Omega\setminus\Gamma, \qquad
\ell \psi=0 \quad\text{on }\partial\Omega, \qquad \psi|_\Gamma=\phi|_\gamma.
\end{equation*}

To solve this problem,we decompose $\psi=\zeta+\tilde\zeta$ where 
\begin{equation}\label{extra}
-\Delta_\Omega \zeta+b\zeta-\mu\rho \zeta=0 \quad\text{in }\Omega\setminus\Gamma, \qquad
\ell \zeta=0 \quad\text{on }\partial\Omega, \qquad \zeta|_\Gamma=\phi|_\gamma,
\end{equation}
\begin{equation*}
-\Delta_\Omega \tilde\zeta+b\zeta-\mu\rho \tilde\zeta=\rho f_2 \quad\text{in }\Omega\setminus\Gamma, \qquad
\ell \zeta=0 \quad\text{on }\partial\Omega, \qquad \zeta|=0 \quad\text{on }\Gamma.
\end{equation*}
If $\mu\in\varrho(S)$, problem \eqref{extra} admits a unique solution for every $\phi\in W_2^2(\omega)$. Then, we
define the operator $B(\mu)\colon W_2^{2,0}(\omega)\to L_2
(\rho,\Omega)$ by letting $B(\mu)\phi$ be the unique solution of 
\eqref{extra}. 
Consequently,
\begin{equation}\label{SolPsiRes}
\psi=B(\mu)\phi+\rR(S)f_2=B(\mu)\rR(T)f_1+\rR(S)f_2.
\end{equation}
Combining \eqref{SolPhiRes} and \eqref{SolPsiRes}, we obtain
\begin{equation*}
    \rR(\cA)=
    \begin{pmatrix}
        \rR(T) &0  \\
    B(\mu)\rR(T)  & \rR(S)
    \end{pmatrix}.
\end{equation*}
Since $B(\mu)$ is bounded on $\varrho(S)$, it follows that $\varrho(T)\cap\varrho(S)\subset\varrho(\cA)$.

Conversely, let $\mu\in\varrho(\cA)$. The first equation in \eqref{SystemT}, \eqref{SystemS} is decoupled from the second one. Therefore, it is uniquely solvable for every $f_1\in L_2(q,\omega)$, and consequently $\mu\in\varrho(T)$. Next, let $f_1=0$. Then $\phi=0$, and the coupling condition yields $\psi|_\Gamma=0$. Consequently,  \eqref{SystemS} becomes $(S-\mu I)\psi=f_2$, which is uniquely solvable for every $f_2\in L_2(\rho,\Omega)$.  Therefore, $\mu\in\varrho(S)$, and hence
\begin{equation*} 
    \varrho(\cA)\subset\varrho(T)\cap\varrho(S). 
\end{equation*} 
Thus,  $\varrho(\cA)=\varrho(T)\cap\varrho(S)$, which is equivalent to $\sigma(\cA)=\sigma(T)\cup\sigma(S)$.
\end{proof}

Let us describe the spectrum of $T$ and $S$. We state the results in Lemma \ref{LemSpectrumT} and Lemma \ref{LemSpectrumS} below. In order to do it, we consider the operator
\begin{multline*}
 H=-\frac 1q\frac{d^2}{dt^2},\quad \dmn H = \Big\{y\in W_2^2(G)\colon   y'(a_1)=0,\dots, y'(a_K)=0,\\y_1(a_0)=\dots= y_K(a_0), \quad \sum_{k=1}^Ky'_k(a_0)=0\Big\}
\end{multline*}
on the metric graph $G$. It is self-adjoint and has compact resolvent \cite[Th. 3.1.1]{BerkolaikoKuchmentBook}. Therefore, the spectrum of $H$ consists of a sequence of isolated real eigenvalues of finite multiplicity accumulating only at $+\infty$.

\begin{lem}\label{LemSpectrumT}
The operator $T$ is self-adjoint, bounded from below and has a noncompact resolvent. As a set,  its spectrum coincides with the spectrum of $H$. Moreover, every spectral point of $T$ is an eigenvalue of infinite multiplicity. Consequently, $\sigma(T)=\sigma_{\mathrm{ess}}(T)$.
\end{lem}

\begin{proof}
Since $\omega=\Gamma\times G$ and the weight $q=q(t)$ does not depend on $s$, the Hilbert space $L_2(q,\omega)$ can be identified with the tensor product \begin{equation*}
L_2(q,\omega)=L_2(\Gamma)\otimes L_2(q,G).
\end{equation*}
Throughout this proof, $I$ denotes the identity operator on $L_2(\Gamma)$. We claim that
\begin{equation*}
T=I\otimes H.
\end{equation*}
Indeed, for any elementary tensor $u(s,t)=a(s)b(t)$, one has
\begin{equation*}
    Tu(s,t)=-\frac1{q(t)}\partial_t^2u(s,t)=a(s)\Bigl(-\frac1{q(t)}\frac{d^2 b}{dt^2}\Bigr)=(I\otimes H)u(s,t),
\end{equation*}
where $a\in L_2(\Gamma)$ and $b\in\dmn H$. Since finite linear combinations of elementary tensors are dense in $L_2(\Gamma)\otimes L_2(q,G)$, the claim follows.

The operator $H$ is self-adjoint in $L_2(q,G)$. Hence, by the standard properties of tensor products of operators, $I\otimes H$ is self-adjoint in
$L_2(\Gamma)\otimes L_2(q,G)$. Consequently, so is $T$ \cite[VIII.10]{ReedSimonI}. Furthermore, the spectrum of a tensor product with the identity operator satisfies
$\sigma(T) = \sigma(I\otimes H) = \sigma(H)$.

Let $\lambda\in\sigma(H)$ and let $y$ be a corresponding eigenfunction. Then
\begin{equation*}
\{a(s)y(t)\colon a\in L_2(\Gamma)\}\subset\ker(T-\lambda I).
\end{equation*}
Hence, every eigenvalue of $T$ has infinite multiplicity. Therefore, the spectrum of $T$ is purely essential.
\end{proof}

For each page $\Omega_k$ of the open book structure, we introduce the operator
\begin{multline*}
  S_k=\rho^{-1}_k(-\Delta_{\Omega_k}+b) \quad\text{in }L_2(\rho_k,\Omega_k),\\
  \dmn S_k=\left\{\psi_k\in W_2^2(\Omega_k)\colon \ell\psi_k=0 \text{ on }\partial\Omega_k,\quad  \psi_k=0 \text{ on }\Gamma \right\}.
\end{multline*} 
The operator $S_k$ describes the spectral problem on $\Omega_k$ isolated from the rest of the structure by the Dirichlet condition on $\Gamma$.

\begin{lem}\label{LemSpectrumS}
The operator $S$ is the direct sum of the operators $S_1,\dots,S_K$:
\begin{equation}\label{SisDirectSum}
S=S_1\oplus\dots\oplus S_K.
\end{equation}
Hence, $\sigma(S)=\sigma(S_1)\cup\dots\cup\sigma(S_K)$.

Moreover, $S$ is self-adjoint, bounded from below, and has compact resolvent. Its spectrum therefore consists of a sequence of real eigenvalues of finite multiplicity accumulating only at $+\infty$.
\end{lem}
\begin{proof}
Decomposition \eqref{SisDirectSum} follows immediately from the definition of $S$. By standard elliptic theory, $S_k$ is self-adjoint, bounded from below and has compact resolvent \cite{EvansBook}.
\end{proof}

 As a consequence of previous results throughout this section, we state the following theorem whose proof is completed in Section \ref{Section5}. For a better understanding, here we briefly sketch the main points of the proof.
 
\begin{thm}\label{ThmEssSpectrumA}
The spectrum of $\cA$ is real, bounded from below, and has no accumulation points except at $+\infty$. Moreover,
\begin{equation*}
\sigma_{\mathrm{ess}}(\cA)=\sigma(T), \qquad \sigma_{\mathrm{disc}}(\cA)=\sigma(S)\setminus\sigma(T).
\end{equation*}
\end{thm}

\begin{proof}
By Lemma~\ref{LemSpectrumT}, $\sigma(T)=\sigma(H)=\sigma_{\mathrm{ess}}(T)$. Furthermore, by Lemma~\ref{LemSpectrumS}, the operator $S$ is self-adjoint with compact resolvent. Hence, Theorem~\ref{ThmSpectrumA} implies that $\sigma(\cA)$ is real, bounded from below, and has no accumulation points except at $+\infty$.

Every point of $\sigma(S)\setminus\sigma(T)$ is an isolated eigenvalue of finite multiplicity and hence belongs to $\sigma_{\mathrm{disc}}(\cA)$. Furthermore, every point $\lambda\in\sigma(T)\setminus\sigma(S)$ is an eigenvalue of infinite multiplicity. Hence, the operator $\cA-\lambda I$ has an infinite-dimensional kernel and therefore cannot be Fredholm. 

The case $\lambda\in\sigma(T)\cap\sigma(S)$ needs an extensive explanation. Indeed, for every $\phi\in\ker(T-\lambda I)$ one has to solve \eqref{SystemS} with $f_2=0$. Its solvability requires certain compatibility conditions. However, these amount to only finitely many linear constraints on the infinite-dimensional eigenspace of $T$. This leads us to assert that the corresponding eigenspace of $\cA$ should also be infinite-dimensional. The proof of this fact, together with a complete description of the associated generalized eigenspaces, will be given in the next section.
\end{proof}

\section{Structure of generalized eigenspaces}\label{Section5}
For a closed operator $A$, let
\begin{equation*}
\cX_\lambda(A)=\bigcup_{k=1}^{\infty}\ker(A-\lambda I)^k
\end{equation*}
denote the generalized eigenspace (or root subspace) corresponding to an eigenvalue $\lambda\in\sigma(A)$. A vector $w\in\cX_\lambda(A)$ is called a generalized eigenvector of rank~$j$ if
\begin{equation*}
w\in\ker(A-\lambda I)^j,\qquad w\notin\ker(A-\lambda I)^{j-1}.
\end{equation*}
We also write $\cE_\lambda(A)=\ker(A-\lambda I)$ for the eigenspace corresponding to $\lambda$.

\subsection{Eigenspaces in the case $\lambda\in\sigma(S)\setminus\sigma(T)$.}
We seek all nontrivial solutions $w=(u,v)$ of the equation $\cA w=\lambda w$, that is,
\begin{equation*}
(T-\lambda I)u=0,\qquad (\mathring{S}-\lambda I)v=0,\qquad v|_\Gamma=u|_\gamma.
\end{equation*}
Since $\lambda\notin\sigma(T)$, the first equation  yields $u=0$. Hence, $v|_\Gamma=0$,  so that $v\in\dmn S$. Therefore, $(S-\lambda I)v=0$. Thus,
the eigenspace $\cE_\lambda(\cA)$ is naturally identified with the eigenspace  $\cE_\lambda(S)$. More precisely,
\begin{equation*}
\cE_\lambda(\cA)=\left\{(0,v)\colon v\in\cE_\lambda(S)\right\}=\{0\}\times\cE_\lambda(S).
\end{equation*}
In particular,
$\dim\cE_\lambda(\cA)<\infty$.

To determine $\cX_\lambda(\cA)$, we examine if an eigenvector $(0,v)\in\cE_\lambda(\cA)$ can be extended to a Jordan chain of length two.
This amounts to finding $(u^*,v^*)$ satisfying
\begin{equation}\label{SystemRank2}
  (T-\lambda I)u^*=0,\qquad   (\mathring S-\lambda I)v^*=v,\qquad   v^*|_\Gamma=u^*|_\gamma.
\end{equation}
Since $\lambda\notin\sigma(T)$, the first equation implies $u^*=0$. Hence $v^*|_\Gamma=0$, so that $v^*\in\dmn S$ and $(S-\lambda I)v^*=v$, where $v\in\cE_\lambda(S)$. The self-adjointness of $S$ yields $\Ran(S-\lambda I)=\cE_\lambda(S)^\perp$, and therefore this equation has no solution. Thus, $\lambda$ is semisimple and $\cX_\lambda(\cA)=\cE_\lambda(\cA)$ for all $\lambda\in\sigma(S)\setminus\sigma(T)$.

\subsection{Eigenspaces in the case $\lambda\in\sigma(T)\setminus\sigma(S)$.}
We now determine which eigenfunctions of $T$ can be extended to eigenvectors of~$\cA$. This requires solving the boundary value problem
\begin{equation}\label{ProblemV}
(-\Delta_{\Omega}+b-\lambda\rho)v=0 \text{ in }\Omega,\qquad \ell v=0 \text{ on }\partial\Omega,\qquad v=u|_\gamma\text{ on }\Gamma.
\end{equation}
The solvability depends on the boundary values of $u$ on $\gamma$. To describe them, let $y^{(1)},\dots,y^{(r)}$ be a basis of $\cE_\lambda(H)$, where $r$ is the multiplicity of $\lambda$ as an eigenvalue of $H$. Since $\cE_\lambda(T)=L_2(\Gamma)\otimes\cE_\lambda(H)$, every eigenfunction of $T$ admits the representation
\begin{equation}\label{UrepreL2}
u=\sum_{j=1}^r\beta_jy^{(j)},\qquad
\beta_j\in L_2(\Gamma).
\end{equation}
For each basis function, we introduce the vector
\begin{equation*}
Y^{(j)}=
\bigl(y^{(j)}(a_1),y^{(j)}(a_2),\dots,y^{(j)}(a_K)\bigr)
\end{equation*}
of its boundary values at the outer vertices of $G$. Consequently,
\begin{equation}\label{TraceU}
u|_\gamma=\sum_{j=1}^r\beta_jY^{(j)}.
\end{equation}

Recall that the coupling condition $v|_\Gamma=u|_\gamma$ is understood in the sense of \eqref{u=v}. In order for problem \eqref{ProblemV} to admit a
solution, its boundary data must satisfy
\begin{equation*}
u|_{\gamma_k}\in W_2^{3/2}(\Gamma),\qquad k=1,\dots,K.
\end{equation*}
By \eqref{TraceU}, this is equivalent to $\beta_1,\dots,\beta_r\in W_2^{3/2}(\Gamma)$.

The above construction defines a bounded linear operator
\begin{equation}\label{OperatorP}
    \cP_\lambda\colon W_2^{3/2}(\Gamma)\otimes\cE_\lambda(H)\to W_2^2(\Omega),
\end{equation}
which assigns to each function $u$ the solution $v=\cP_\lambda u$ of problem \eqref{ProblemV}.

Consequently,
\begin{equation*}
\cE_\lambda(\cA)=\big\{(u,\cP_\lambda u)\colon u\in W_2^{3/2}(\Gamma)\otimes\cE_\lambda(H)\big\}.
\end{equation*}
Since $W_2^{3/2}(\Gamma)$ is infinite-dimensional, the tensor product $W_2^{3/2}(\Gamma)\otimes\cE_\lambda(H)$ is infinite-dimensional. Hence so is $\cE_\lambda(\cA)$.

There are no generalized eigenvectors of rank $2$ for the eigenvalue $\lambda$. Indeed, let $(u^*,v^*)$ be a generalized eigenvector corresponding to an eigenvector $(u,v)$. Then
\begin{equation}\label{TUStarU}
(T-\lambda I)u^*=u.
\end{equation}
Since $u\in\cE_\lambda(T)$ and $T$ is self-adjoint, this equation is solvable only if $u=0$. However, $\cE_\lambda(\cA)$ is the graph of the operator $\cP_\lambda$, so that $u=0$ implies $v=\cP_\lambda0=0$, contradicting the fact that $(u,v)$ is an eigenvector. Therefore, $\cX_\lambda(\cA)=\cE_\lambda(\cA)$.

\subsection{Eigenspaces in the case $\lambda\in\sigma(T)\cap\sigma(S)$.}
The structure of the eigenspace changes significantly when $\lambda$ belongs to both spectra. In contrast to the previous cases, the coupling condition gives rise to nontrivial compatibility constraints, leading to a more intricate description of the eigenspace. Moreover, this is the case in which nontrivial Jordan chains may appear.

Every eigenfunction $v$ of $S$ determines an eigenvector $w=(0,v)$ of $\cA$. Therefore,
\begin{equation*}
\{0\}\times\cE_\lambda(S)\subset\cE_\lambda(\cA).
\end{equation*}
An eigenfunction $u$ of $T$, however, extends to an eigenvector of $\cA$ only if problem \eqref{ProblemV} is solvable. Thus, only a proper subspace of $\cE_\lambda(T)$ contributes to $\cE_\lambda(\cA)$.

Let $\lambda$ be an eigenvalue of $S$ of multiplicity $n$, and let 
\begin{equation}\label{BasisS}
    \big\{V^{(1)},\dots,V^{(n)}\big\}
\end{equation}
be an orthonormal basis of $\cE_\lambda(S)$.

\begin{lem}\label{PropSolvabilityBook}
Given a function $f\in L_2(\rho,\Omega)$ and a vector-valued function
$$g=(g_1,\dots,g_K)\in \big(W_2^{3/2}(\Gamma)\big)^K,$$
consider the  boundary value problem
\begin{equation}\label{ProblemBookSolvability}
-\Delta_{\Omega}z+bz-\lambda\rho z=\rho f \text{ in }\Omega,\qquad \ell z=0 \text{ on }\partial\Omega,\qquad
z=g \text{ on }\Gamma.
\end{equation}
Then the problem admits a solution $z\in W_2^2(\Omega)$ if and only if
\begin{equation}\label{BookSolvabilityCond}
 \int_\Gamma g\cdot\partial_\nu V^{(i)}\,d\Gamma + \int_{\Omega} \rho f\,V^{(i)}\,d\mu=0,
\qquad
i=1,\dots,n.
\end{equation}
Here $\partial_\nu V^{(i)}=\bigl(\partial_{\nu_1}V^{(i)}_1,\dots,\partial_{\nu_K}V^{(i)}_K\bigr)$ denotes the vector of normal derivative on $\Gamma$, where
$\nu_k$ is the inward unit normal to the page $\Omega_k$.
\end{lem}

\begin{proof}
The statement follows from the Fredholm alternative for the
self-adjoint operator $S$ with compact resolvent.
Conditions \eqref{BookSolvabilityCond} are obtained by multiplying
the equations in \eqref{ProblemBookSolvability} by $V^{(i)}$,
integrating by parts over each page $\Omega_k$, and using the
boundary conditions on $\Gamma_k=\partial\Omega_k\setminus\Gamma$ and $\Gamma$.
\end{proof}

We now characterize those eigenfunctions of $T$ that can be lifted to eigenvectors of $\cA$. To this end, consider the space
\begin{equation*}
\cN_\lambda=\spn\big\{\partial_\nu V^{(1)},\dots,\partial_\nu V^{(n)}\big\}\subset L_2(\Gamma)^K.
\end{equation*}
By \eqref{BookSolvabilityCond},  problems \eqref{ProblemV} are solvable if and only if the trace $u|_\gamma$ belongs to the orthogonal complement of $\cN_\lambda$ in $L_2(\Gamma)^K$. This motivates the definition

\begin{equation*}
\cT_\lambda=\Bigl\{u\in W_2^{3/2}(\Gamma)\otimes\cE_\lambda(H)\colon
u|_\gamma\in\cN_\lambda^\perp\Bigr\}.
\end{equation*}

For every $u\in\cT_\lambda$, the corresponding solution of \eqref{ProblemV} differ by elements of $\cE_\lambda(S)$. We denote by $\cQ_\lambda u$ the unique solution orthogonal to $\cE_\lambda(S)$. This defines a bounded linear operator
\begin{equation}\label{OperatorQ}
\cQ_\lambda\colon\cT_\lambda\to W_2^2(\Omega)\cap\cE_\lambda(S)^\perp.
\end{equation}
Thus, the eigenspace splits into two natural components:
\begin{equation*}
\cE_\lambda(\cA)=\graph\cQ_\lambda
\oplus
\bigl(\{0\}\times\cE_\lambda(S)\bigr).
\end{equation*}
Equivalently,
\begin{equation*}
\cE_\lambda(\cA)=\bigl\{(u,\cQ_\lambda u+v)\colon u\in\mathcal T_\lambda,\; v\in\cE_\lambda(S)\bigr\}.
\end{equation*}

\subsection{Root subspaces in the case $\lambda\in\sigma(T)\cap\sigma(S)$.}
In contrast to the previous cases, the root subspace $\cX_\lambda(\cA)$ need not coincide with the eigenspace. Since \eqref{TUStarU} is solvable only when
$u=0$, every Jordan chain, if it exists, has the form
\begin{equation}\label{Chain}
(0,v)\mapsto (u^*,v^*),
\end{equation}
where $u^*\in\cE_\lambda(T)$. It remains to determine which eigenvectors $(0,v)$ admit such generalized eigenvectors.

Consider an eigenvector $v\in\cE_\lambda(S)$ and represent it with respect to the orthonormal basis \eqref{BasisS}:
\begin{equation}\label{RepreV}
v=\sum_{j=1}^{n}\alpha_jV^{(j)}, \qquad \alpha_i\in\Real.
\end{equation}
To construct a generalized eigenvector $(u^*,v^*)$, we have to solve
\begin{equation*}
 (\mathring S-\lambda I)v^*=v,\qquad   v^*|_\Gamma=u^*|_\gamma.
\end{equation*}
Accordingly, we restrict ourselves to eigenfunctions $u^*\in\cE_\lambda(T)$ whose representation \eqref{UrepreL2} has
coefficients $\beta_j\in W_2^{3/2}(\Gamma)$.
Then $v^*$ is determined from the boundary value problem
\begin{equation}\label{JordanProblem}
-\Delta v^*+bv^*-\lambda\rho v^*=
\rho\sum_{j=1}^{n}\alpha_jV^{(j)}
\quad\text{in }\Omega,\qquad
\ell v^*=0\quad\text{on }\partial\Omega,\qquad
v^*|_\Gamma=u^*|_\gamma.
\end{equation}

Using the orthonormality of the basis \eqref{BasisS}, Lemma~\ref{PropSolvabilityBook} shows the problem is solvable if and only if
\begin{equation*}
\alpha_i=-\int_\Gamma u^*|_\gamma\cdot\partial_\nu V^{(i)}\,d\Gamma, \qquad i=1,\dots,n.
\end{equation*}
For convenience, and in view of the representation \eqref{TraceU}, we introduce the $n\times r$ matrix-valued function $\Phi$ on $\Gamma$ by
\begin{equation*}
\Phi_{ij}(s)=Y^{(j)}\cdot\partial_\nu V^{(i)}(s),
\qquad s\in\Gamma.
\end{equation*}
This leads to the operator $M_\lambda\colon\bigl(W_2^{3/2}(\Gamma)\bigr)^r\to\Comp^n$,
given by
\begin{equation}\label{OperatorM}
M_\lambda\beta=-\int_\Gamma\Phi\beta\,d\Gamma.
\end{equation}
With this notation, the solvability conditions for \eqref{JordanProblem} take the form
\begin{equation}\label{AMB}
\alpha=M_\lambda\beta,
\end{equation}
where $\alpha=(\alpha_1,\dots,\alpha_n)$ and $\beta=(\beta_1,\dots,\beta_r)$. 

Hence, an eigenvector $(0,v)\in\cE_\lambda(\cA)$, where $v$ is represented by \eqref{RepreV}, initiates a Jordan chain \eqref{Chain}  if and only if $\alpha\in\Ran M_\lambda$. If the vectors $\alpha$ and $\beta$ are related by \eqref{AMB}, then
\begin{equation*}
u^*=\sum_{j=1}^{r}\beta_jy^{(j)},
\end{equation*}
while $v^*$ is obtained as the corresponding solution of \eqref{JordanProblem}, $v^*\in\cE_\lambda(S)^\perp$. Thus, every Jordan chain is uniquely determined by a pair
$(\alpha,\beta)$ satisfying \eqref{AMB} with $\alpha\neq0$, and conversely, every such pair generates a Jordan chain.

It remains to show that $\cA$ has no generalized eigenvectors of rank $3$.
Suppose that $(\hat{u},\hat{v})$ is such a vector. Then
\begin{equation*}
(\cA-\lambda I)\begin{pmatrix}\hat{u}\\ \hat{v}\end{pmatrix}
=\begin{pmatrix}u^*\\ v^*\end{pmatrix},
\end{equation*}
where $(u^*,v^*)$ is a generalized eigenvector of rank $2$. In particular, $(T-\lambda I)\hat{u}=u^*$. Since $T$ is self-adjoint and $u^*\in\cE_\lambda(T)$, this equation is solvable only if $u^*=0$. Consequently, $(S-\lambda I)\hat{v}=v^*$. Applying the same argument to the self-adjoint ope\-rator $S$, we conclude that $v=0$, contradicting the fact that $(0,v)$ is an eigenvector of $\cA$.

Let $\cJ_\lambda(\cA)$ denote the linear span of all generalized eigenvectors of rank $2$ corresponding to $\lambda$. Then
\begin{equation*}
\cX_\lambda(\cA)=\cE_\lambda(\cA)\oplus\cJ_\lambda(\cA).
\end{equation*}

The above analysis yields the following characterization of the Jordan structure.
\begin{lem}
Let $\lambda\in\sigma(T)\cap\sigma(S)$. The eigenvalue $\lambda$ has exactly $\rank M_\lambda$ Jordan blocks of length $2$, while all remaining Jordan blocks are of size $1$.
\end{lem}

The  lemma shows that the existence of Jordan chains is determined by the structure of $\cE_\lambda(H)$ and $\cE_\lambda(S)$. More precisely, it depends on the interaction between the boundary values $Y^{(j)}$ of eigenfunctions of $H$ and the normal derivatives of  eigenfunctions of $S$. If these two families do not interact, then $M_\lambda=0$, and consequently the eigenvalue $\lambda$ is semisimple, even if $\lambda\in\sigma(T)\cap\sigma(S)$.

An eigenfunction of $H$ may vanish identically on some edges of the graph $G$, as illustrated by the model example in Section~\ref{Section6}; see also Fig.~\ref{FigExample}. In this case, the corresponding components of $Y^{(j)}$ are zero, so that the support of $y^{(j)}$ determines the set of edge indices corresponding to the nonzero components of $Y^{(j)}$. Since $S=S_1\oplus\cdots\oplus S_K$,
the basis of $\cE_\lambda(S)$ may be chosen so that each basis function is supported on a single page. We may therefore associate each  function $V^{(i)}$ with a unique page~$\Omega_k$. Suppose that the supports of the eigenfunctions $y^{(1)},\dots,y^{(r)}$ are contained in edges whose indices are disjoint from the indices of the pages supporting the basis functions of $\cE_\lambda(S)$. Then $\Phi\equiv0$, and hence $M_\lambda=0$.

We summarize the above results in the following theorem.

\begin{thm}\label{MainThm}
Let $\lambda\in\sigma(\cA)$ be an eigenvalue. Nontrivial Jordan chains can occur only if $\lambda$ is a common eigenvalue of the operators $T$ and $S$. If $\lambda$ is an eigenvalue of either $T$ or $S$, but not both, then $\cX_\lambda(\cA)=\cE_\lambda(\cA)$, that is, the eigenvalue $\lambda$ is semisimple.
More precisely,
\begin{enumerate}
\item if $\lambda\in\sigma(S)\setminus\sigma(T)$, then $\cE_\lambda(\cA)=\{0\}\times\cE_\lambda(S)$ and $\dim\cE_\lambda(\cA)<\infty$;
\item if $\lambda\in\sigma(T)\setminus\sigma(S)$, then $\cE_\lambda(\cA)=\graph\cP_\lambda$ and $\dim\cE_\lambda(\cA)=\infty$,
where $\cP_\lambda\colon W_2^{3/2}(\Gamma)\otimes\cE_\lambda(H)\to W_2^2(\Omega)$ is the operator defined in \eqref{OperatorP}. 
\end{enumerate}

If $\lambda$ is a common eigenvalue of the operators $T$ and $S$, then
\begin{equation*}
\cE_\lambda(\cA)=\graph\cQ_\lambda\oplus\bigl(\{0\}\times\cE_\lambda(S)\bigr)\quad\text{and}\quad\dim\cE_\lambda(\cA)=\infty,
\end{equation*}
where $\cQ_\lambda$ is the operator defined in \eqref{OperatorQ}. Moreover,
\begin{equation*}
\cX_\lambda(\cA)=\cE_\lambda(\cA)\oplus\cJ_\lambda(\cA),
\end{equation*}
where $\cJ_\lambda(\cA)$ is the subspace spanned by the generalized eigenvectors of rank $2$, and 
\begin{equation*}
\dim\bigl(\cX_\lambda(\cA)/\cE_\lambda(\cA)\bigr)=\rank M_\lambda,
\end{equation*}
where $M_\lambda$ is the operator defined by \eqref{OperatorM}. 
\end{thm}

The spectral structure of the limit operator $\cA$ established in this paper provides the foundation for the asymptotic analysis of oscillatory systems on open-book structures with singular density perturbations. In contrast to the perturbed operators, whose spectra consist of isolated eigenvalues of finite multiplicity, the limit operator may possess eigenvalues of infinite multiplicity and nontrivial Jordan chains. As a consequence, the bifurcation of multiple eigenvalues exhibits several qualitatively different scenarios. A principal objective of our subsequent work is to derive asymptotic expansions for the eigenvalues and eigenfunctions of the perturbed problem and to describe the convergence of the corresponding spectral subspaces to the root subspaces of the limit operator.

\section{A model example}\label{Section6}
In this section, we consider an example in which the eigenspace and the Jordan chains corresponding to a common eigenvalue of $S$ and $T$ can be computed explicitly.
We specialize  problem \eqref{LimitEqV}--\eqref{LimitEqU} to an open-book structure consisting of three identical pages, see Fig.~\ref{FigExample}. Each page is identified with the rectangle 
\begin{equation*}
 \Omega_k=\left(0,\tfrac{\pi}{2}\right)\times(0,1), \qquad k=1,2,3,
\end{equation*}
equipped with the Euclidean metric, where $(x_1,x_2)$ are Cartesian coordinates. Thus, $\Delta_{\Omega_k}=\partial_{x_1}^2+\partial_{x_2}^2$. 
The Dirichlet condition is imposed on the binding
\begin{equation*}
\Gamma=\{(0,x_2)\colon 0<x_2<1\}.
\end{equation*}
On the remaining part of the boundary, the operator $\ell$ is subject to the Neumann conditions
\begin{equation*}
\partial_{x_1}v(\tfrac{\pi}{2},x_2)=0,\quad \partial_{x_2}v(x_1,0)=0,\quad \partial_{x_2}v(x_1,1)=0.
\end{equation*}

The potential is assumed to vanish identically, $b=0$, and the mass density is constant, $\rho=1$. 
We also consider the open-book structure $\omega=\Gamma\times G$, where $G$ is the star graph with three edges of unit length. Finally, the weight $q$ is chosen to be constant and equal to $\pi^2/4$.

\begin{figure}[!h]
  \centering
  \includegraphics[scale=0.4]{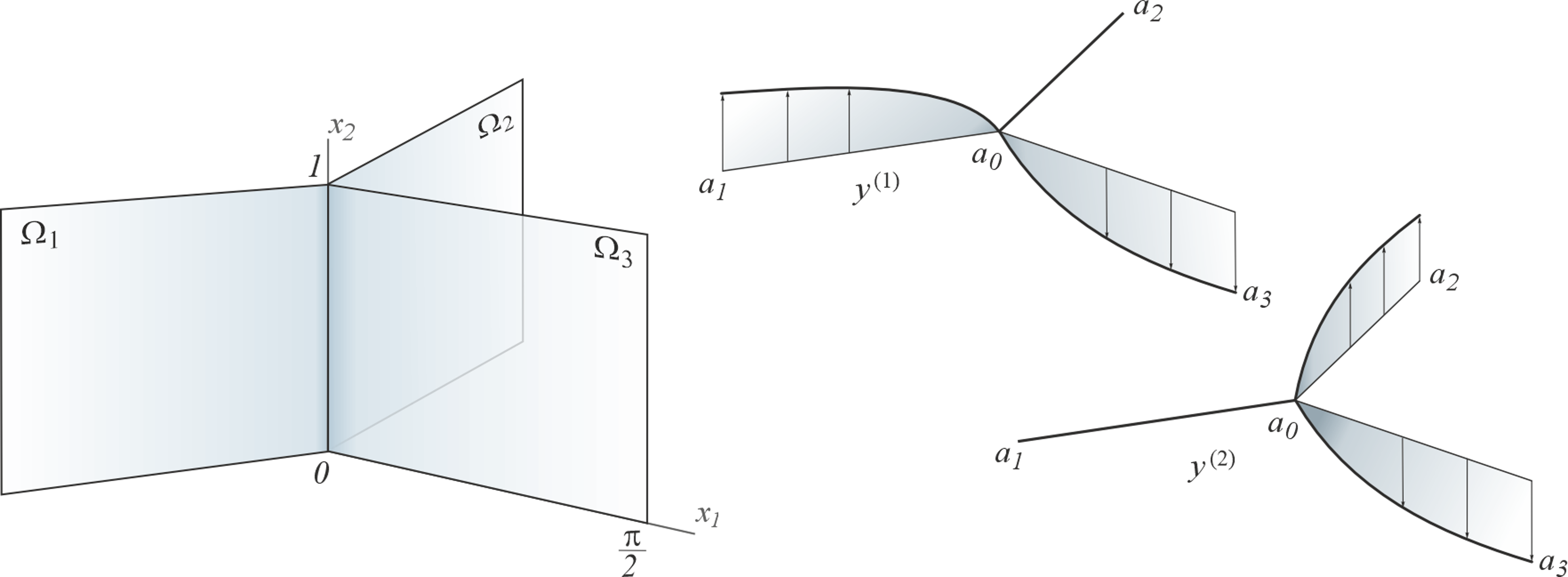}\\
  \caption{The model example. Left: the open-book structure $\Omega$. Right: the
basis eigenfunctions $y^{(1)}$ and $y^{(2)}$ of the graph operator $H$
corresponding to the double eigenvalue $\lambda=1$.}\label{FigExample}
\end{figure}

For the above choice of the open-book geometry, potentials, and mass densities, the number $\lambda=1$ is an eigenvalue of $\cA$. The
corresponding spectral data are as follows (see the proof below).

\noindent
\textit{Basis of $\cE_1(S)$:}
    \begin{equation}\label{BaseV}
     V^{(1)}=\left(\tfrac{2}{\sqrt{\pi}}\sin x_1,0,0\right),\quad
     V^{(2)}=\left(0,\tfrac{2}{\sqrt{\pi}}\sin x_1,0\right),\quad
     V^{(3)}=\left(0,0,\tfrac{2}{\sqrt{\pi}}\sin x_1\right).
    \end{equation}
    
\medskip    
\noindent    
\textit{Basis of $\cE_1(H)$:}
    \begin{equation}\label{BaseY}
        y^{(1)}=\left(\sin\frac{\pi t_1}{2},0,-\sin\frac{\pi t_3}{2}\right),\qquad
        y^{(2)}=\left(0,\sin\frac{\pi t_2}{2},-\sin\frac{\pi t_3}{2}\right),
    \end{equation}
where $t_j\in[0,1]$, see Fig.~\ref{FigExample}. 

\medskip 
\noindent 
\textit{Eigenspace of $T$:} 
\begin{equation*}
  \cE_1(T)=\big\{\beta_1(x_2)y^{(1)}(t)+\beta_2(x_2)y^{(2)}(t)\colon \beta_1, \beta_2\in L_2(0,1) \big\}.
\end{equation*}

\medskip 
\noindent 
\textit{Eigenfunctions of $T$ admitting extension to eigenvectors of $\cA$:}
\begin{equation}\label{spaceTl}
    \cT_1=\left\{\beta_1y^{(1)}+\beta_2y^{(2)}\colon \int_0^1\beta_j\,dx_2=0,\;j=1,2\right\}.
\end{equation}

\medskip    
\noindent    
\textit{Basis of $\cE_1(\cA)$:}
    \begin{equation*}
        (0,V^{(1)}),\quad (0,V^{(2)}),\quad (0,V^{(3)}),\quad (u_j,\cQ_1u_j),\quad j\in\mathbb N,
    \end{equation*}
where $\{u_j\}_{j=1}^\infty$ is an arbitrary basis of $\cT_1$.

\medskip    
\noindent    
\textit{Jordan chains:}
\begin{equation*}
  \left(0,V^{(3)}-V^{(1)}\right)\longmapsto \left(\frac{\sqrt{\pi}}{2}y^{(1)},z^{(1)}\right),\quad
  \left(0,V^{(3)}-V^{(2)}\right)\longmapsto
\left(\frac{\sqrt{\pi}}{2}y^{(2)},z^{(2)}\right),
\end{equation*}
where 
\begin{align*}
&z^{(1)}=\left(
\left(\tfrac{x_1}{\sqrt{\pi}}-\tfrac{\sqrt{\pi}}{2}\right)\cos x_1,\,
0,\,
\left(\tfrac{\sqrt{\pi}}{2}-\tfrac{x_1}{\sqrt{\pi}}\right)\cos x_1
\right),\\
&z^{(2)}=\left(
0,\,
\left(\tfrac{x_1}{\sqrt{\pi}}-\tfrac{\sqrt{\pi}}{2}\right)\cos x_1,\,
\left(\tfrac{\sqrt{\pi}}{2}-\tfrac{x_1}{\sqrt{\pi}}\right)\cos x_1
\right).
\end{align*}

We now explain how the above formulas are obtained. Owing to the Neumann boundary conditions on the outer boundaries of the pages, the relevant
boundary-value problems on $\Omega_k$ reduce to ordinary differential equations on $(0,\pi/2)$ with homogeneous Dirichlet condition at $x_1=0$ and homogeneous Neumann condition at $x_1=\pi/2$. The corresponding problems on the graph are
one-dimensional from the outset. Consequently, all spectral objects can be computed explicitly.

Along the binding $\Gamma$, the normal derivative coincides with the partial derivative with respect to $x_1$. Therefore, at $x_1=0$,
\begin{equation*}
    \partial_\nu V^{(1)}=\left(\tfrac{2}{\sqrt{\pi}},0,0\right),\qquad
    \partial_\nu V^{(2)}=\left(0,\tfrac{2}{\sqrt{\pi}},0\right),\qquad
    \partial_\nu V^{(3)}=\left(0,0,\tfrac{2}{\sqrt{\pi}}\right).
\end{equation*}
Hence, the space $\cN_1$ is three-dimensional. The traces of the basis functions \eqref{BaseY} at the outer vertices of the graph are $Y^{(1)}=(1,0,-1)$ and $Y^{(2)}=(0,1,-1)$. Hence, 
\begin{equation*}
u|_\gamma=\beta_1Y^{(1)}+\beta_2Y^{(2)}=(\beta_1,\beta_2,-\beta_1-\beta_2).
\end{equation*}
Then the condition
$u|_\gamma\in\cN_1^\perp$ is equivalent to
\begin{equation*}
\int_0^1\beta_1\,dx_2=0,\qquad
\int_0^1\beta_2\,dx_2=0,\qquad \int_0^1(\beta_1-\beta_2)\,dx_2=0,
\end{equation*}
which is precisely the characterization of $\cT_1$ in \eqref{spaceTl}.  In this case,
\begin{equation*}
M_1\begin{pmatrix}
     \beta_1 \\
     \beta_2 
   \end{pmatrix}
=\left(
-\frac{2}{\sqrt{\pi}}\int_0^1\beta_1\,dx_2,\,
-\frac{2}{\sqrt{\pi}}\int_0^1\beta_2\,dx_2,\,
\frac{2}{\sqrt{\pi}}\int_0^1(\beta_1+\beta_2)\,dx_2
\right)^t.
\end{equation*}
Choosing $\beta^{(1)}=\left(\tfrac{\sqrt{\pi}}{2},0\right)$ and $\beta^{(2)}=\left(0,\tfrac{\sqrt{\pi}}{2}\right)$, we obtain
\begin{equation*}
M_1\beta^{(1)}=(-1,0,1),\qquad M_1\beta^{(2)}=(0,-1,1).
\end{equation*}
Since these vectors are linearly independent, $\rank M_1=2$. 

Hence, the eigenvalue $\lambda=1$ admits two Jordan chains, starting from the eigenvectors $\left(0,V^{(3)}-V^{(1)}\right)$ and $\left(0,V^{(3)}-V^{(2)}\right)$.
The second components $z^{(i)}$ of the generalized eigenvectors satisfy  problem \eqref{JordanProblem}, which in this case takes the form
\begin{equation*}
\Delta z^{(i)}+z^{(i)}=V^{(3)}-V^{(i)}\quad\text{in }\Omega,\quad
\ell z^{(i)}=0\quad\text{on }\partial\Omega,\quad z^{(i)}|_\Gamma=\tfrac{\sqrt{\pi}}{2}Y^{(i)}.
\end{equation*}
The boundary conditions and the right-hand side suggest seeking solutions independent of $x_2$. Thus, setting $z^{(i)}(x_1,x_2)=z^{(i)}(x_1)$ reduces the problem to an ordinary differential equation. The nonzero components of $z^{(i)}$ on the pages $\Omega_i$ and $\Omega_3$ differ only by sign and satisfy
\begin{equation*}
\frac{d^2 z}{dx_1^2}+z=\frac{2}{\sqrt{\pi}}\sin x_1,\quad
z(0)=-\frac{\sqrt{\pi}}{2},\quad
z'\!\left(\frac{\pi}{2}\right)=0,
\end{equation*}
while the remaining component vanishes identically.

This example demonstrates the construction of the eigenspace of $\cA$ and the corresponding Jordan chains associated with a common eigenvalue of $S$ and $T$.

\bibliographystyle{amsplain}

\end{document}